\documentclass[a4paper,12pt]{amsart}
\usepackage[english]{babel}
\usepackage{amsmath,amsthm,amssymb,amsfonts}
\usepackage{multicol}

\usepackage[pdftex]{color}
\usepackage[bookmarks=true,hyperindex,pdftex,colorlinks, citecolor=blue,linkcolor=blue, urlcolor=blue]{hyperref}
\usepackage[dvipsnames]{xcolor}
\usepackage{mathtools}
\usepackage{graphicx}
\usepackage[toc]{appendix}
\usepackage{tikz}
 
\usepackage[normalem]{ulem}

\usetikzlibrary{matrix}

\newtheorem{theorem}{Theorem}[section]
\newtheorem*{theoremA}{Theorem A}
\newtheorem*{theoremB}{Theorem B}

\newtheorem{lemma}[theorem]{Lemma}

\newtheorem{proposition}[theorem]{Proposition}
\newtheorem{corollary}[theorem]{Corollary}

\theoremstyle{definition}
\newtheorem{definition}[theorem]{Definition}

\theoremstyle{remark}

\numberwithin{equation}{section}

\newcommand{\R}{\mathbb{R}}

\newcommand{\N}{\mathbb{N}}

\DeclareMathOperator{\diam}{diam\,}
\DeclareMathOperator{\co}{co}

\DeclareMathOperator{\id}{Id}
\DeclareMathOperator{\Iso}{Iso}
\DeclareMathOperator{\dens}{dens}
\DeclareMathOperator{\ran}{ran}
\DeclareMathOperator{\Span}{span}
\DeclareMathOperator{\Sign}{sign}

\newcommand{\nn}[1]{{\left\vert\kern-0.25ex\left\vert\kern-0.25ex\left\vert #1 
		\right\vert\kern-0.25ex\right\vert\kern-0.25ex\right\vert}}

\renewcommand{\geq}{\geqslant}
\renewcommand{\leq}{\leqslant}

\newcommand{\pten}{\ensuremath{\widehat{\otimes}_\pi}}

\newcommand{\eps}{\varepsilon}

\begin{document}

\title{Daugavet points in projective tensor products}

\author[Dantas]{Sheldon Dantas}
\address[Dantas]{Departament de Matemàtiques and Institut Universitari de Matemàtiques i Aplicacions de Castelló (IMAC), Universitat Jaume I, Campus del Riu Sec. s/n, 12071 Castelló, Spain. \newline
\href{http://orcid.org/0000-0001-8117-3760}{ORCID: \texttt{0000-0001-8117-3760} } }
\email{\texttt{dantas@uji.es}}

\author[Jung]{Mingu Jung}
\address[Jung]{Basic Science Research Institute and Department of Mathematics, POSTECH, Pohang 790-784, Republic of Korea \newline
\href{http://orcid.org/0000-0000-0000-0000}{ORCID: \texttt{0000-0003-2240-2855} }}
\email{\texttt{jmingoo@postech.ac.kr}}

\author[Rueda Zoca]{Abraham Rueda Zoca}
\address[Rueda Zoca]{Universidad de Murcia, Departamento de Matem\'aticas, Campus de Espinardo 30100 Murcia, Spain
 \newline
\href{https://orcid.org/0000-0003-0718-1353}{ORCID: \texttt{0000-0003-0718-1353} }}
\email{\texttt{abraham.rueda@um.es}}
\urladdr{\url{https://arzenglish.wordpress.com}}

\begin{abstract} In this paper, we are interested in studying when an element $z$ in the projective tensor product $X \pten Y$ turns out to be a Daugavet point. We prove first that, under some hypothesis, the assumption of $X \pten Y$ having the Daugavet property implies the existence of a great amount of isometries from $Y$ into $X^*$. Having this in mind, we provide methods for constructing non-trivial Daugavet points in $X \pten Y$. We show that $C(K)$-spaces are examples of Banach spaces such that the set of the Daugavet points in $C(K) \pten Y$ is weakly dense when $Y$ is a subspace of $C(K)^*$. Finally, we present some natural results on when an elementary tensor $x \otimes y$ is a Daugavet point. 
\end{abstract}

\thanks{Sheldon Dantas was supported by the project OPVVV CAAS CZ.02.1.01/0.0/0.0/16\_019/0000778 and by the Estonian Research Council grant PRG877. Mingu Jung was supported by the Basic Science Research Program through the National Research Foundation of Korea (NRF) funded by the Ministry of Education (NRF-2019R1A2C1003857). Abraham Rueda Zoca was supported by was supported by Juan de la Cierva-Formaci\'on fellowship FJC2019-039973, by MTM2017-86182-P (Government of Spain, AEI/FEDER, EU), by MICINN (Spain) Grant PGC2018-093794-B-I00 (MCIU, AEI, FEDER, UE), by Fundaci\'on S\'eneca, ACyT Regi\'on de Murcia grant 20797/PI/18, by Junta de Andaluc\'ia Grant A-FQM-484-UGR18 and by Junta de Andaluc\'ia Grant FQM-0185.}


\maketitle

\thispagestyle{plain}

\section{Introduction}

In 2001, D. Werner asked if the Daugavet property is stable under tensor products (see \cite[Section 6, Question 3]{W}). More specifically, he asked whether the projective tensor product between two Banach spaces $X \pten Y$ has the Daugavet property whenever $X$ and $Y$ satisfy such a property. It turns out that this is the case whenever $X$ and $Y$ are (isometric) preduals of $L_1$ with the Daugavet property (see \cite{RTV} and more recently \cite{MR}). Nevertheless, in its complete generality, this question seems to be still open. In the present paper, we are interested in the ``localised'' version of problem for Daugavet {\it points}. We say that a norm-one element $x$ is a \emph{Daugavet point} if given a slice $S$ of the unit ball $B_X$, then there exists $y \in S$ satisfying $\|x - y\| \approx 2$ (see \cite{AHLP}). By \cite[Corollary 2.3]{W}, the Banach space $X$ has the Daugavet property if and only if every element of the unit sphere $S_X$ is a Daugavet point. Therefore, the concept of Daugavet points can be faced as a local version of the Daugavet property. See \cite{AHLP,ALMT,HPV,JR} for background on Daugavet points and diverse relevant examples.

Coming back to tensor products, there are natural ways of constructing elementary tensors $x \otimes y$ which are Daugavet points (see Proposition \ref{elementarytensors}). The question here is, therefore, when we can produce tensors (not necessarily elementary) which are Daugavet points. In other words, collecting all the information we have given so far and putting them together, the following question seems to arrive naturally.

\begin{center}
{\it When an element $z \in S_{X \pten Y}$ can be a Daugavet point?}
\end{center}

In order to tackle this problem, let us take a brief moment to explain that the assumption that there are plenty of these points in $X \pten Y$ implies severe restrictions between the Banach spaces $X$ and $Y$. Since a Banach space $X$ with the Daugavet property satisfies the strong diameter two property and it has an octahedral norm (by the proof of \cite[Lemma 3]{S} and by \cite[Lemma 2.8]{KSSW} respectively), the results from Section 3 of \cite{LLR} suggest us that whenever $Y$ is a separable (uniformly convex) Banach space and $X \pten Y$ has the strong diameter 2 property, then $Y$ is isometric to a subspace of $X^*$. With this in mind, it is natural to think that there exist isometries from $Y$ into $X^*$ whenever $X \pten Y$ has the Daugavet property. Proving this conjecture is our first aim in the paper. In what follows, the set $\Iso(Y, X^*)$ is the set of all isometries from $Y$ into $X^*$ (not necessarily surjetive). 

\begin{theoremA} Let $X$ be a Banach space and suppose that $Y$ is a separable uniformly convex Banach space. Assume further that $X \pten Y$ has the Daugavet property. Then, the set $\Iso(Y, X^*)$ is $w^*$-dense in $B_{\mathcal{L}(Y, X^*)}$.	
\end{theoremA}

Let us notice that Theorem A not only guarantees the existence of an isometry from $Y$ into $X^*$ when $Y$ is a separable uniformly convex Banach space, but also that the isometries between these spaces are abundant. It is worth mentioning that, up to our knowledge, it is not known whether $X \pten Y$ has the Daugavet property when $X$ is an $L_1$-predual and $Y$ is a finite-dimensional Banach space different from $\ell_1^{\dim(Y)}$.

Under some hypothesis it is possible to present a (theoretical) converse of Theorem A (see Proposition \ref{DPtensors2}). Nevertheless, what we will be interested in is to know when an isometry from $Y$ into $X^*$ yields Daugavet points. Indeed, we have the following result which gives us a way of constructing non-trivial Daugavet points in $X \pten Y$.

\begin{theoremB} Let $z = \sum_{i=1}^n \lambda_i x_i \otimes y_i \in \co(S_X \otimes S_Y)$. Assume that
	\begin{itemize}
	\setlength\itemsep{0.2cm} 
		\item[(1)] $X$ is an $L_1$-predual with the Daugavet property,
		\item[(2)] $\displaystyle \left\| \sum_{j=1}^m \mu_j x_j \right\| = \sum_{j=1}^m |\mu_j|$ for every $\mu_1, \ldots, \mu_m \in \R$, and
		\item[(3)] There exists an isometry $\phi$ from $Y$ into $X^*$. 
	\end{itemize}
	Then, $z \in S_{X \pten Y}$ is a Daugavet point.
\end{theoremB}

Let us notice that besides the existence of an isometry from $Y$ into $X^*$, the rest of assumptions in Theorem B are required just on the Banach space $X$.

Finally, in Theorem \ref{theo:densiC(K)} and Corollary \ref{C(K)result} we provide an example of a Banach space where Theorem B can be applied, namely $C(K)$-spaces, to produce examples of tensors $C(K)\pten Y$ for which, when $Y$ is isometric to a subspace of $C(K)^*$, the set of Daugavet points is even weakly dense (see Corollary \ref{cor:consefinC(K)}).

Finally, in the way of looking for examples where our main results apply, we discovered that if $X$ is an $L$-embedded space with the metric approximation property and the Daugavet property then $X$ has the \textit{weak operator Daugavet property} (see Definition \ref{def:WODP}) if $\dens(X)\leq \omega_1$. Since the proof is a tricky and technical, and since we think it of interest by itself, we devote Section \ref{sectionanexample} to the proof of this fact.

Before presenting our results, let us give the necessary background the reader needs to following the present paper. All of the Banach spaces throughout the paper are consider to be {\it real}. The closed unit ball and the unit sphere of a Banach space $X$ are denoted by $B_X$ and $S_X$, respectively. We denote by $\mathcal{L}(X, Y)$ the set of all bounded linear operators from $X$ into $Y$. When $X = Y$, we simply denote $\mathcal{L}(X, Y)$ by $\mathcal{L}(X)$. The symbol $\mathcal{B}(X \times Y)$ stands for the bounded bilinear forms from $X \times Y$ into $\R$. We denote by $\mathcal{F}(X, Y)$ all the finite rank operators from $X$ into $Y$. We define a {\it slice} of the unit ball of $X$ by the set
\begin{equation*}
S(B_X, x^*, \alpha) := \left\{ x \in B_X: x^*(x) > \|x^*\| - \alpha \right\},
\end{equation*}
where $x^* \in X^*$ and $\alpha > 0$. We write shortly $S(x^*, \alpha)$ when the space $X$ is understood from the context. 
We denote the convex hull of a set $A$ by $\co(A)$ and its closure $\overline{\co}(A)$.

We say that a Banach space $X$ has the \emph{Daugavet property} if every rank-one operator $T \in \mathcal{L}(X)$ satisfies the equality 
\begin{equation*}
\|\id + T\| = 1 + \|T\|, 
\end{equation*} 
where $\id$ is the identity operator on $X$. By \cite[Lemma 2.1]{KSSW}, a Banach space $X$ has the Daugavet property if and only if for every $\eps > 0$, every point $x \in S_X$, and every slice $S$ of $B_X$, there exists $y \in S$ such that $\|x + y\| > 2 - \eps$. We will use this geometric characterization in the proof of Theorem \ref{DPtensors2}. 
 
Recall that the \emph{projective tensor product}, denoted by $X \pten Y$, of Banach spaces $X$ and $Y$ is the completion of the algebraic tensor product $X \otimes Y$ under the following norm 
\[
\|u \| = \inf \left\{ \sum_{i=1}^n \|x_i \| \|y_i\| : u = \sum_{i=1}^n x_i \otimes y_i \right\}. 
\]
It is well known that $(X \pten Y)^* = \mathcal{B}( X \times Y) = \mathcal{L} (Y, X^*)$ and that $B_{X \pten Y} = \overline{\co} (B_X \otimes B_Y) = \overline{\co} (S_X \otimes S_Y)$ where we write the sets $B_X \otimes B_Y$ and $S_X \otimes S_Y$ for the sets $\{ x \otimes y : x \in B_X, y \in B_Y\}$ and $\{x \otimes y : x \in S_X, y \in S_Y\}$, respectively. We refer the reader to the books \cite{DF,rya} for background on projective tensor products.

\section{The Results}

We start this section by proving Theorem A. 
In fact, what we shall prove is a more general result which says that if $X \pten Y$ satisfies the Daugavet property, the amount of elements in $\mathcal{L}(Y, X^*)$ which attain their norms at every strongly exposed point of $S_Y$ is quite big when $Y^*$ is assumed to be separable and $Y$ has the Radon-Nikod\'ym property (for short, RNP). Let us use $w^*$ to denote the weak-star topology.

\begin{theorem} \label{DPtensors1} Let $X$ and $Y$ be Banach spaces. Assume that $Y^*$ separable and $Y$ has RNP. If $X \pten Y$ has the Daugavet property, then the set
\begin{equation*}
\mathcal{C} := \Big\{ T \in S_{\mathcal{L}(Y, X^*)}: \|T(y)\| = 1 \ \mbox{for every strongly exposed point} \ y \in S_Y \Big\}
\end{equation*}
is $w^*$-dense in $B_{\mathcal{L}(Y, X^*)}$.
\end{theorem}

\begin{proof} We will follow the ideas of \cite[Lemma 2.12]{KSSW}. Let $W$ be a $w^*$-open subset of $B_{\mathcal{L}(Y, X^*)}$. We will find $T \in \mathcal{C} \cap W$. Since $Y^*$ is a separable space, we may take $\{y_n^*: n \in \N\}$ to be a dense subset of $S_{Y^*}$. Fix $x^* \in S_{X^*}$. Since $X \pten Y$ has the Daugavet property by applying successively \cite[Lemma 3]{S}, there exists a sequence of non-empty $w^*$-open $(W_n)$ with the following properties:

\begin{itemize}
\setlength\itemsep{0.2cm} 
	\item[(a)] $\overline{W_1}^{w^*} \subseteq W$, 
	\item[(b)] $\overline{W_{n+1}}^{w^*} \subseteq W_n$ for $n \in \N$, and 
	\item[(c)] for every $G \in W_n$ and for every $i=1, \ldots, n$, 
	\begin{equation*}
	\|G + x^* \otimes y_i^*\| > 2 - \frac{1}{n}.
	\end{equation*}
 
\end{itemize}
Now, take $T \in \bigcap_{n=1}^{\infty} \overline{W_n}^{w^*} \subseteq W$. Then, we have that $T \in W$ and, by using (b), that $T \in W_n$ for every $n \in \N$. In particular, $\|T\| \leq 1$. The theorem will be established once we can show the following claim.

\vspace{0.2cm} 
\noindent 
{\it Claim:} $T$ attains its norm at every strongly exposed point $y \in S_Y$. 
\vspace{0.2cm}

Let $y_0 \in S_Y$ be an arbitrary strongly exposed point of $B_Y$. By using (c) above, we have that $\|T + x^* \otimes y_n^*\| = 2$ for every $n \in \N$, which implies that $\|T + x^* \otimes y^*\| = 2$ for every $y^* \in S_{Y^*}$. In particular, $\|T + x^* \otimes y_0^*\| = 2$. This implies that there exists a sequence $(y_n) \subset S_Y$ such that
\begin{equation*}
2 - \frac{1}{n} \leq \|T(y_n) + y_0^*(y_n) x^* \| \leq \|T(y_n)\| + |y_0^*(y_n)|
\end{equation*} 
and this implies that $\|T(y_n)\| \longrightarrow 1$ and $|y_0^*(y_n)| \longrightarrow 1$ as $n \rightarrow \infty$. Let $\theta_n \in \{1, -1\}$ so that $y_0^*(\theta_n y_n) = |y_0^*(y_n)|$ for every $n \in \N$. Then, $\theta_n y_n$ converges to $y_0$ since $y_0$ is a strongly exposed point, and this shows that $T$ attains its norm at $y_0$, as we wanted.	
\end{proof}

In what follows (see Proposition \ref{DPtensors2}), we present a converse to Theorem A when the Banach space $X$ satisfies the so-called weak operator Daugavet property (WODP, for short). We are forced to highlight here that Proposition \ref{DPtensors2} will not influence the rest of the paper but from our point of view it has a profound theoretical meaning. In order to prove it, let us present the definition of WODP.

\begin{definition}\cite[Definition 5.2]{MR}\label{def:WODP} We say that the Banach space $X$ has the {\it weak operator Daugavet property} (WODP, for short) if given $\eps > 0$, $x_1, \ldots, x_n \in S_X$, a slice $S$ of $B_X$, and $x' \in B_X$, we can find $x \in S$ and $T \in \mathcal{L}(X)$ such that
	\begin{itemize}
			\setlength\itemsep{0.2cm}
		\item[(a)] $\|T\| \leq 1 + \eps$,
		\item[(b)] $\|T(x_i) - x_i\| < \eps$ for every $i \in \{1, \ldots, n\}$, and 
		\item[(c)] $\|T(x) - x'\| < \eps$.
	\end{itemize}
\end{definition} 
\noindent
We have that if a Banach space $X$ satisfies the WODP, then it satisfies the Daugavet property (see \cite[Remark 5.3]{MR}). Let us present a list of Banach spaces which satisfy such a property. In what follows, MAP stands for the metric approximation property.
\begin{itemize}
\setlength\itemsep{0.2cm} 
	\item[(a)] $L_1$-preduals with the Daugavet property,
	\item[(b)] $L_1(\mu, Z)$-spaces for $\mu$ atomless and $Z$ an arbitrary Banach space,
	\item[(c)] $W_1 \pten W_2$, where $W_1$ and $W_2$ satisfy (a) or (b), and
	\item[(d)] $L$-embedded spaces $X$ which satisfies the MAP and Daugavet property when $\dens(X) \leq \omega_1$.
\end{itemize}
For (a), (b), and (c) we refer to reader to the papers \cite{MR,RTV}. It seems that item (d) is not known in the literature and we present a proof of it in Section \ref{sectionanexample}. Now we are ready to prove Proposition \ref{DPtensors2}

\begin{proposition} \label{DPtensors2} Let $X$ and $Y$ be Banach spaces. Assume that $X$ satisfies the WODP. Suppose further that $\Iso (Y, X^*)$ is norming. Then, $X \pten Y$ has the Daugavet property. 
\end{proposition}

\begin{proof} In order to prove that $X \pten Y$ has the Daugavet property, let us take $z \in S_{X \pten Y}$, $\eps > 0$, and an arbitrary slice $S := S(B_{X \pten Y}, G, \alpha)$ of $B_{X \pten Y}$ with $\|G\| = 1$ and $\alpha > 0$. We will find $x \otimes y \in S$ such that 
\begin{equation*} 
\|z + x \otimes y\| > \frac{2 - \eps}{1 + \eps}. 
\end{equation*} 
For this, consider $x_0 \otimes y \in S$ with $x_0 \in S_X$ and $y \in S_Y$. Let us find $\sum_{i=1}^N x_i \otimes y_i$ so that
\begin{equation} \label{ineq3}
\left\|z - \sum_{i=1}^N x_i \otimes y_i \right\| < \frac{\eps}{12(1 + \eps)}.
\end{equation}
Since $\Iso(Y, X^*)$ is norming, we may find an isometry $\phi \in \mathcal{L}(Y, X^*)$ to be such that
\begin{equation} \label{ineq4}
\sum_{i=1}^N \phi(y_i)(x_i) > 1 - \frac{\eps}{12}
\end{equation}
and, as $\phi$ is an isometry, we can find $u \in B_X$ such that
\begin{equation} \label{ineq5}
\phi(y)(u) > 1 - \frac{\eps}{12}.
\end{equation}
Note that the set $\{z\in B_X: G(z,y)>1-\alpha\}$ is nonempty and defines a slice of $B_X$. Since $X$ has the WODP, we can find an element $x \in \{z\in B_X: G(z,y)>1-\alpha\}$ (which implies that $x\otimes y\in S$) and an operator $T \in \mathcal{L}(X)$ such that
\begin{equation} \label{ineqs}
\|T\| \leq 1 + \widetilde{\eps}, \ \ \ \|T(x) - u\| < \widetilde{\eps}, \ \ \ \mbox{and} \ \ \ \|T(x_i) - x_i\| < \widetilde{\eps},
\end{equation}
for every $i = 1, \ldots, N$, where
\begin{equation*}
0 < \widetilde{\eps} < \min \left\{ \frac{\eps}{12}, \frac{\eps}{6 \sum_{i=1}^N \|y_i\|}\right\}.
\end{equation*}
Define now $B \in \mathcal{B}(Y \times X)$ by $B(v, w) := \phi(v)(T(w))$ for every $v \in Y$ and $w \in X$. Then, $\|B\| \leq 1 + \eps$ and, by using (\ref{ineq5}) and (\ref{ineqs}), we have that 
\begin{equation} \label{ineq6}
B(y, x) = \phi(y)(T(x)) \geq \phi(y)(u) - \|T(x) - u \| > 1 - \frac{\eps}{6}.
\end{equation}
Moreover,
\begin{eqnarray*}
\sum_{i=1}^N B(y_i, x_i) = \sum_{i=1}^N \phi(y_i)(T(x_i)) &=& \sum_{i=1}^N \phi(y_i)(T(x_i) - x_i + x_i) \\
&\stackrel{(\ref{ineqs})}{\geq}& \sum_{i=1}^N \phi(y_i)(x_i) - \widetilde{\eps} \sum_{i=1}^N \|y_i\|	\\
&\stackrel{(\ref{ineq4})}{>}& 1 - \frac{\eps}{4}. 
\end{eqnarray*} 
Therefore, using this last estimate, we have that 
\begin{eqnarray*}
\|z + x \otimes y\| 
&\geq& \left\| \sum_{i=1}^N x_i \otimes y_i + x \otimes y \right\| - \left\|z - \sum_{i=1}^N x_i \otimes y_i \right\| \\
&\stackrel{(\ref{ineq3})}{>}& \left\| \sum_{i=1}^N x_i \otimes y_i + x \otimes y \right\| - \frac{\eps}{12(1 + \eps)} \\
&>& \frac{1}{\|B\|} B \left( \sum_{i=1}^N x_i \otimes y_i + x \otimes y \right) - \frac{\eps}{12(1 + \eps)} \\
&\stackrel{(\ref{ineq5}), (\ref{ineq6})}{>}& \frac{1}{1 + \eps} \left(1 - \frac{\eps}{4} + 1 - \frac{\eps}{6} \right) - \frac{\eps}{12(1 + \eps)} \\
&>& \frac{2 - \eps}{1 + \eps}.
\end{eqnarray*}
\end{proof}


Arguing similarly to Theorem \ref{DPtensors2}, it is possible to present another (again theoretical) sufficient condition so that $X \pten Y$ satisfies the Daugavet property. 

\begin{theorem} \label{DPtensors3} Let $X$ be a Banach space with the WODP and $Y$ be a Banach space with the Radon-Nikod\'ym property. If the set
\begin{equation*}
\mathcal{C} := \Big\{ T \in S_{\mathcal{L}(Y, X^*)}: \|T(y)\| = 1 \ \mbox{for every strongly exposed point} \ y \in S_Y \Big\}
\end{equation*}
is norming. Then, $X \pten Y$ has the Daugavet property.
\end{theorem}

By Theorem \ref{DPtensors2}, we can see that the assumption that the set $\Iso(Y, X^*)$ is norming has a strong connection with the Daugavet property in the projective tensor product $X \pten Y$. We will see next that the existence of at least one isometry already implies the existence of a Daugavet point $z \in S_{X \pten Y}$ under WODP assumption. This is our first result which shows a way of constructing non-trivial Daugavet points in the projective tensor product.

\begin{theorem} \label{Dpfortensors3} Let $X$ and $Y$ be a Banach spaces. Assume that $X$ satisfies the WODP. Given an element $z = \sum_{i=1}^n \lambda_i x_i \otimes y_i \in \co(S_X \otimes S_Y)$, suppose that there exists an isometry $\phi \in \mathcal{L}(Y, X^*)$ so that $\phi(y_i)(x_i) = 1$ for every $i = 1, \ldots, n$. Then, $z \in S_{X \pten Y}$ is a Daugavet point.
\end{theorem}

\begin{proof} Let $\eps > 0$ be given and let us fix a slice $S := S(B_{X \pten Y}, G, \alpha)$ of $B_{X \pten Y}$ with $\|G\| = 1$ and $\alpha > 0$. Let $x_0 \otimes y_0 \in S$ with $x_0 \in S_X$, $y_0 \in S_Y$. Since $\phi$ is an isometry, we may take $u_0 \in B_X$ such that
	\begin{equation*}
	\phi(y_0)(u_0) > 1 - \eps. 
	\end{equation*}
	Since $X$ has the WODP, we can find $T \in \mathcal{L}(X)$ and $u_1 \in \{z\in B_X: G(z,y_0)>1-\alpha\}$ (which implies $u_1\otimes y_0\in S)$ such that 
	\begin{equation*}
	\|T\| \leq 1 + \eps, \ \ \ \|T(u_1) - u_0\| < \eps, \ \ \ \mbox{and} \ \ \ \ \|T(x_i) - x_i\| < \eps, 
	\end{equation*} 	
	for every $i \in \{1, \ldots, n\}$. Define now $B \in \mathcal{B}(Y \times X)$ by $B(y, x) := \phi(y)(T(x))$ for every $y \in Y$ and $x \in X$. Then, $\|B\| \leq 1 + \eps$. Moreover,
	\begin{eqnarray*}
		B(y_0, u_1) = \phi(y_0)(T(u_1)) &=& \phi(y_0)(T(u_1) - u_0 + u_0) \\
		&\geq& \phi(y_0)(u_0) - \|T(u_1) - u_0\| \\
		&>& 1 - 2 \eps. 	
	\end{eqnarray*} 
	Then, since $\phi(y_i)(x_i) = 1$ for every $i=1,\ldots, n$, we have that
	\begin{eqnarray*}
		B(z + u_1 \otimes y_0) 
		&=& \sum_{i=1}^n \lambda_i \phi(y_i)(T(x_i)) + B(y_0, u_1) \\
		&>& \sum_{i=1}^n \lambda_i( \phi(y_i)(x_i) - \|T(x_i) - x_i\|) + 1 - 2 \eps \\
		&>& \sum_{i=1}^n \lambda_i \phi(y_i)(x_i) - \eps +1 - 2 \eps	\\
		&=& 2 - 3 \eps. 
	\end{eqnarray*}	
	Therefore
	\begin{equation*}
	\|z + u_1 \otimes y_0\| \geq \|B\|^{-1}(2 - 3 \eps) > \frac{2 - 3\eps}{1 + \eps}.
	\end{equation*} 
	This proves that $z \in S_{X \pten Y}$ is a Daugavet point.	
\end{proof}

We prove now Theorem B. It gives us a condition so that an element $z = \sum_{i=1}^n \lambda_i x_i \otimes y_i \in \co(S_X \otimes S_Y)$ is a Daugavet point under the hypothesis that $X$ is an $L_1$-predual with the Daugavet property and $Y$ is an isometric subspace of $X^*$. Comparing this to Theorem \ref{Dpfortensors3}, we now assume that $\{x_1, \ldots, x_n\}$ is equivalent to the canonical basis of $\ell_1^n$.

\begin{proof}[Proof of Theorem B] Let $S := S(B_{X \pten Y}, G, \alpha)$ be a non-empty slice and $\eps > 0$ be given, where $G \in (X \pten Y)^* = \mathcal{L}(X, Y^*)$ and $\alpha > 0$. Let $z = \sum_{i=1}^n \lambda_i x_i \otimes y_i \in \co(S_X \otimes S_Y)$. We will find an element $x \otimes y \in S$ with
	\begin{equation} \label{dp2} 
	\|z + x \otimes y\| > 2 - \eps.
	\end{equation}	
	Let us consider an arbitrary $x_0 \otimes y \in S$ and let $\delta > 0$ be small enough so that 
	\[
	\frac{2(1-\delta)^2}{1+\delta} > 2 - \eps.
	\]
	
	Since $\phi \in \mathcal{L}(Y, X^*)$ is an isometry, we can find $u_0, u_1, \ldots, u_n \in S_X$ such that 
	\begin{equation} \label{dp3} 
	\phi(y_i)(u_i) > 1 - \frac{\delta}{n} \ \ \ \mbox{and} \ \ \ \phi(y)(u_0) > 1 - \delta. 
	\end{equation}
	Consider the set $\{z\in B_X: G(z,y)>1-\alpha\}$, which defines a slice of $B_X$. Since $X$ has the Daugavet property (see \cite[Lemma 2.8]{KSSW}), there exists $x \in \{z\in B_X: G(z,y)>1-\alpha\}$ (which implies $x\otimes y\in S$) such that
	\begin{equation} \label{dp4} 
	\|y + \mu x \| \geq (1 - \delta)(\|y\| + |\mu|)
	\end{equation}
	for every $y \in \Span \{x_1, \ldots, x_n\}$ and $\mu \in \R$. Define $\varphi \in \mathcal{L}( \Span\{x_1,\ldots, x_n,x\}, X)$ by
	\begin{equation*}
	\varphi \left( \sum_{j=1}^m \mu_j x_j + \mu x \right) := \sum_{j=1}^m \mu_j u_j + \mu u_0.
	\end{equation*}
	Then, for every $\sum_{j=1}^m \mu_j x_j + \mu x \in \Span \{x_1, \ldots, x_n, x\}$, we have 
	\begin{eqnarray*}
		\left\| \varphi \left( \sum_{j=1}^m \mu_j x_j + \mu x \right) \right\| 
		&\leq& \sum_{j=1}^m |\mu_j| + |\mu| \\
		&=& \left\| \sum_{j=1}^m \mu_j x_j \right\| + |\mu| \\
		&\stackrel{(\ref{dp4})}{\leq }& \frac{1}{1 - \delta} \left\| \sum_{j=1}^m \mu_j x_j + \mu x \right\|. 
	\end{eqnarray*} 
	This shows that $\|\varphi\| \leq \frac{1}{1 - \delta}$. Since $\varphi$ is a compact operator and $X$ is an $L_1$-predual, we can consider $\varphi$ defined on the whole $X$ with $\| \varphi \| \leq \frac{1+\delta}{1-\delta}$ (see \cite[Theorem 6.1, (3)]{L}). Now, define $T \in (X \pten Y)^* = \mathcal{L}(X, Y^*)$ by
	\begin{equation*}
	T(u \otimes v) := \phi(v)(\varphi(u)) \ \ (u \in X, v \in Y).
	\end{equation*}
	So, we have that $\Vert T\Vert\leq \frac{1+\delta}{1-\delta}$ and that $T (x \otimes y) = \phi (y) (u_0) > 1 - \delta$. Therefore,
	\begin{eqnarray*}
		\|z + x \otimes y\| \geq \frac{1-\delta}{1+\delta}(T(z + x \otimes y)) 
		&\stackrel{(\ref{dp3})}{>}& \frac{1-\delta}{1+\delta} \left( \sum_{i=1}^n \lambda_i \left(1 - \frac{\delta}{n} \right) + (1 - \delta) \right) \\
		&=& \frac{1-\delta}{1+\delta} \left(1 - \frac{\delta}{n} + 1 - \delta \right) \\
		&>& \frac{1-\delta}{1+\delta}(2 - 2 \delta) \\
		&>& 2- \eps,
	\end{eqnarray*} 	
	so (\ref{dp2}) is established. 
\end{proof}

Let us exhibit examples where Theorem B applies to find projectve tensor products with a weakly dense set of Daugavet points. To do so, let us start with the following lemma.

\begin{lemma}
	Let $K$ be a compact Hausdorff topological space with no isolated point. Let $f_1,\ldots, f_n\in S_{\mathcal C(K)}$ and $A\subseteq K$ be a non-empty open subset. Then there are $g_1,\ldots, g_n\in S_{\mathcal C(K)}$ such that
	\begin{enumerate}
	\setlength\itemsep{0.2cm} 
		\item $g_i=f_i$ on $K\setminus A$.
		\item $\displaystyle \left\| \sum_{i=1}^n \lambda_i g_i \right\| =\sum_{i=1}^n \vert \lambda_i\vert$ holds for every $\lambda_1,\ldots, \lambda_n\in\mathbb R$.
	\end{enumerate}
\end{lemma}

\begin{proof}
	Let $\mathcal P:=\{\sigma=(\sigma_1,\ldots, \sigma_n): \sigma_i\in\{-1,1\}\}$.
	
	Since $A$ is a non-empty open set in $K$, which does not contain any isolated point, then $A$ is infinite, so take $\{t_\sigma: \sigma\in \mathcal P\}\subseteq A$ different points. Since they are different points we can find open set $V_\sigma$ such that $t_\sigma\in V_\sigma\subseteq \overline{V_\sigma}\subseteq A$ and such that $\overline{V_\sigma}\cap \overline{V_\nu}=\emptyset$ if $\sigma\neq \nu\in\mathcal P$. By Urysohn's lemma we can find, for every $\sigma\in\mathcal P$, a function $0\leq f_\sigma\leq 1$ with $f_\sigma=0$ on $K\setminus V_\sigma$ and $f_\sigma(t_\sigma)=1$. Also find by Urysohn's lemma a function $0\leq h\leq 1$ such that $h=1$ on $K\setminus A$ and $h=0$ on $\bigcup\limits_{\sigma\in \mathcal P} \overline{V_\sigma}$. Define, for every $i\in\{1,\ldots, n\}$, a continuous function $g_i:K\longrightarrow \mathbb R$ by the equation
	$$g_i:=\left(1-\sum_{\sigma\in\mathcal P}\sigma_i f_\sigma\right)f_i h+\sum_{\sigma\in\mathcal P}\sigma_i f_\sigma.$$
	Let us prove that $\Vert g_i\Vert\leq 1$. To this end, pick $t\in K$. Then we have two possibilities:
	\begin{enumerate}
		\item If $t\in V_\sigma$ for some (unique) $\sigma$, then $h(t)=0$ and so 
		$$\vert g_i(t)\vert=\vert \sigma_i f_\sigma(t)\vert=\vert \sigma_i\vert \vert f_\sigma(t)\vert\leq \Vert f_\sigma\Vert_\infty=1.$$
		\item If $t\notin \bigcup\limits_{\sigma\in \mathcal P}V_\sigma$ then $\sum_{\sigma\in\mathcal P}\sigma_i f_\sigma(t)=0$ and so
		$$\vert g_i(t)\vert=\vert f_i(t)\vert\vert h_i(t)\vert\leq \Vert f_i\Vert=1.$$
	\end{enumerate}
	Taking maxima on $K$ we get that $\Vert g_i\Vert\leq 1$. Moreover, since $h=1$ and $\sum_{\sigma\in\mathcal P}\sigma_i f_\sigma=0$ on $K\setminus A$ we get that $g_i=f_i$ on $K\setminus A$, and (1) is proved.
	
	Let us finally prove (2). To this end pick any $\lambda_1,\ldots, \lambda_n\in\mathbb R$ and define $\sigma:=(\Sign(\lambda_1),\ldots, \Sign(\lambda_n))$. Then, by definition
	$$g_i(t_\sigma)=\sigma_if_\sigma(t_\sigma)
	=\Sign(\lambda_i).$$
	So
	\[
	\left\Vert \sum_{i=1}^n \lambda_i g_i\right\Vert \geq \sum_{i=1}^n \lambda_i g_i(t_\sigma)=\sum_{i=1}^n \lambda_i \Sign (\lambda_i)=\sum_{i=1}^n \vert \lambda_i\vert.
	\]
	which proves (2).
\end{proof}

Now we are able to prove the following theorem.

\begin{theorem}\label{theo:densiC(K)}
	Let $K$ be a compact Hausdorff topological space with no isolated point. Pick $f_1,\ldots, f_n\in S_{\mathcal C(K)}$. Then, for every $i\in\{1,\ldots, n\}$, there is a sequence $\{g_k^i\}\subseteq S_{\mathcal C(K)}$ which weakly converges to $f_i$ for every $i$ and such that $\Vert \sum_{j=1}^n \lambda_j g_k^j\Vert=\sum_{j=1}^n \vert \lambda_j\vert$ holds for every $\lambda_1,\ldots, \lambda_n\in\mathbb R$ and every $k\in\mathbb N$.
\end{theorem}

\begin{proof}
	Since $K$ does not have any isolated point we can find a sequence of non-empty open subsets $\{V_n\}\subseteq K$ being pairwise disjoint. Now, apply the previous lemma to $f_1,\ldots, f_n$ and $V_k$ to find $g_k^1,\ldots, g_k^n$ satisfying the thesis of the lemma, and we only have to prove that $g_k^i\rightarrow f_i$ weakly. To this end, notice that $f_i-g_k^i$ is a sequence on $k$ which is pairwise disjoint. Hence $f_i-g_k^i$ is a bounded sequence which converges pointwise to $0$. By Rainwater theorem \cite[Corollary 3.61]{FHHMPZ}, $f_i-g_k^i\rightarrow 0$ weakly, and we are done.\end{proof}

Now we are ready to prove the following result.

\begin{corollary} \label{C(K)result} 
	Let $K$ be a compact Hausdorff topological space without any isolated point. Let $Y$ be a Banach space. Then, given any convex combination of slices $C$, there exists an element $z\in C$ such that $z=\sum_{i=1}^n \lambda_i x_i\otimes y_i\in co(S_{C(K)}\otimes S_Y)$ with
\begin{equation*} 
\left\| \sum_{j=1}^n \mu_j x_j \right\| = \sum_{j=1}^n |\mu_j|,
\end{equation*}
for every $\mu_1,\ldots, \mu_n\in\mathbb R$.
\end{corollary}

\begin{proof}
	Pick a convex combination of slices $C=\sum_{i=1}^n \lambda_i S_i$ and pick an element $\sum_{i=1}^n \lambda_i u_i\otimes y_i\in C$. Apply Theorem \ref{theo:densiC(K)} to find, for every $i$ a sequence $x_k^i$ which is weakly convergent to $u_i$ and such that the $\{x_k^i: 1\leq i\leq n\}$ is isometrically equivalent to the basis of $\ell_1^n$. Notice that $x_k^i\otimes y_i$ weakly converges to $u_i\otimes y_i$, so find $k$ large enough so that $x_k^i\otimes y_i\in S_i$ holds for every $1\leq i\leq n$ and define $x_i:=x_k^i$ to finish the proof.
\end{proof}

If we combine Theorem B and Corollary \ref{C(K)result} we get the desired consequence.

\begin{corollary}\label{cor:consefinC(K)}
Let $K$ be a compact Hausdorff topological space without any isolated point. Let $Y$ be a Banach space such that $Y$ is isometric to a subspace of $C(K)^*$. Then any convex combination of slices $C$ of $B_{C(K)\pten Y}$ contains a Daugavet point. In particular, the set of Daugavet points is weakly dense in $B_{C(K)\pten Y}$.
\end{corollary}

\begin{proof}
All but the last assertion is clear. Moreover, the last part follows by the well-known Bourgain's lemma \cite[Lemma II.1]{GGMS}, which asserts that every nom-empty relatively weakly open subset of the unit ball contains a convex combination of slices of the unit ball.
\end{proof}

Note that even if a Banach space $X$ satisfies that the set of all Daugavet points is weakly dense then it is not true that every non-empty relatively weakly open subset of $B_X$ has diameter two. Indeed, in \cite[Section 4]{ALMT} it is exhibited an example of Banach space $Z$ having a weakly dense set of Daugavet points but such that the unit ball has points of weak-to-norm continuity and, in particular, the unit ball of $Z$ contains arbitrarily small diameter non-empty relatively weakly open subsets. Notice, however, that the assumptions of Theorem B does not permit this phenomenon. Indeed, under the assumptions of Theorem B, $X^*$ is an $L_1$-space, then the condition (3) in Theorem B implies that $\mathcal{L}(Y, X^*)$ has an octahedral norm (see \cite[Theorem 3.2]{LR}) and then $X \pten Y$ satisfies that even every convex combination of slices of the unit ball has diameter $2$ (see \cite[Theorem 2.1]{BLR}).

We finish this section by presenting a short but detailed study of elementary tensors $x \otimes y$ in $X \pten Y$ which are Daugavet points. This is motivated by the recent paper \cite{LP}, where the authors prove that if $x \in S_X$ is a $\Delta$-point, then $x \otimes y \in S_{X \pten Y}$ is a $\Delta$-point for every $y \in S_Y$ (see \cite[Remark 5.4]{LP}). We recall that $x \in S_X$ is a {\it $\Delta$-point} if given $\eps > 0$ and a slice $S$ of $B_X$ containing $x$, then there exists $y \in S$ satisfying $\|x - y\| \geq 2 - \eps$. We prove a converse of this result when we assume that $y$ is a strongly exposed point. We give the analogous for Daugavet points. All of these results are summed up in the Proposition \ref{elementarytensors} below.

Before presenting its proof, let us just make some brief comments. Notice that $x \in S_X$ is a Daugavet point if and only $-x \in S_X$ is a Daugavet point. Thus, we can substitute the condition $\|x - y\| \geq 2 - \eps$ for $\|x + y\| \geq 2 - \eps$ whenever it is convenient. We will be using this fact in the proof of Proposition \ref{elementarytensors}.(b) without any explicit reference (this does not hold for $\Delta$-points). Moreover, in the proof of Proposition \ref{elementarytensors}.(c), we need to guarantee that if $x_0 \otimes y_0 \in S_{X \pten Y}$ is a Daugavet point then, for every slice $S$ of $B_{X\pten Y}$ and every $\varepsilon>0$, there exists an {\it elementary} tensor $u_0 \otimes v_0\in S$ such that $\|x_0 \otimes y_0 - u_0 \otimes v_0\| \geq 2 - \eps$. For this, we highlight the following lemma observed in \cite[Remark 2.3]{JR}, which guarantees such a fact.

\begin{lemma} \label{lemmaslices} Let $X$ be a Banach space. Let $x \in S_X$ be a Daugavet point. Then, for each $\eps > 0$ and each slice $S(B_X, x^*, \alpha)$, there exists a new slice $S(B_X, x_0^*, \alpha_0)$ such that $S(B_X, x_0^*, \alpha_0) \subseteq S(B_X, x^*, \alpha)$ and $\|x + y\| \geq 2 - \eps$	for every $y \in S(B_X, x_0^*, \alpha_0)$.
\end{lemma}

Now we are ready to prove the promised result.

\begin{proposition} \label{elementarytensors} Let $X$ and $Y$ be Banach spaces. Let $x_0 \in S_X$ and $y_0 \in S_Y$. 
\begin{itemize}
\setlength\itemsep{0.2cm} 
	\item[(a)] If $y_0 \in S_Y$ is strongly exposed and $x_0 \otimes y_0$ is a $\Delta$-point, then $x_0$ is a $\Delta$-point. 
	
	\item[(b)] If $x_0$ and $y_0$ are both Daugavet points, then so is $x_0 \otimes y_0$. 
	
	\item[(c)] If $y_0$ is denting and $x_0 \otimes y_0$ is Daugavet, then $x_0$ is a Daugavet point.
\end{itemize}
	
\end{proposition}

\begin{proof} (a). In order to prove that $x_0$ is a $\Delta$-point, let us consider $S(B_X, x_0^*, \alpha)$ to be a slice of $B_X$ containing $x_0$ with $\|x_0^*\| = 1$ and $\alpha > 0$. Let $\eps > 0$ be given. Assuming that $y_0 \in S_Y$ is strongly exposed, consider the functional $y_0^* \in S_{Y^*}$ which strongly exposes $y_0$. Then, there exists $\delta = \delta(\eps) > 0$ such that $\diam(S(B_Y, y_0^*, \delta)) < \eps$. Consider $0 < \eta < \delta$. By \cite[Lemma 1.4]{IK}, there exists $x_1^* \in S_{X^*}$ such that $x_0 \in S(B_X, x_1^*, \eta) \subseteq S(B_X, x_0^*, \alpha)$. Now, consider the slice $S_1:=S(B_{X \pten Y}, x_1^* \otimes y_0^*, \eta)$ of $B_{X \pten Y}$. Since $y_0^*$ attains its norm at $y_0$, we have that
	\begin{equation*}
	\langle x_0 \otimes y_0, x_1^* \otimes y_0^* \rangle = x_1^*(x_0) > 1 - \eta,
	\end{equation*}
	that is, $x_0 \otimes y_0 \in S_1$. Since $x_0 \otimes y_0$ is a $\Delta$-point, by using \cite[Lemma 2.1]{JR}, we can guarantee the existence of an elementary tensor $u_0 \otimes v_0 \in S_1$ so that $\|x_0 \otimes y_0 - u_0 \otimes v_0\| \geq 2 - \eps$. In particular,
	\begin{equation} \label{ineq2}
	\|x_0 - u_0\| + \|y_0 - v_0\| \geq 2 - \eps. 
	\end{equation} 
	Since $u_0 \otimes v_0 \in S_1$, we have that $x_1^*(u_0) > 1 - \eta$ and $y_0^*(v_0) > 1 - \eta > 1 - \delta$. Therefore, $u_0 \in S(B_X, x_1^*, \eta) \subseteq S(B_X, x_0^*, \alpha)$ and $v_0 \in S(B_Y, y_0^*, \delta)$. So, $\|y_0 - v_0\| < \eps$ and by (\ref{ineq2}), $\|x_0 - u_0\| \geq 2 - 2 \eps$. This proves that $x_0 \in S_X$ is a $\Delta$-point.	
	
\noindent
(b). Suppose that $x_0 \in S_X$ and $y_0 \in S_Y$ are both Daugavet points. Let $\eps > 0$ be given. Let us fix an arbitrary slice $S := S(B_{X \pten Y}, G, \alpha)$ of $B_{X \pten Y}$ with $\|G\| = 1$ and $\alpha > 0$. Take $(x_1, y_1) \in B_X \times B_Y$ to be such that 
\begin{equation*}
G(x_1, y_1) > 1 - \frac{\alpha}{2}
\end{equation*}
and let us consider the slice of $B_X$
\begin{equation*}
S_1 := \left\{ x \in B_X: G(x, y_1) > \sup_{u \in B_X} G(u, y_1) - \frac{\alpha}{4} \right\}.
\end{equation*}
Since $x_0 \in S_X$ is a Daugavet point, we can find $u_0 \in S_1$ such that $\|x_0 - u_0\| \geq 2 - \eps$. On the other hand, let us consider the slice of $B_Y$
\begin{equation*}
S_2 := \left\{ y \in B_Y: G(u_0, y) > \sup_{v \in B_Y} G(u_0, v) - \frac{\alpha}{4} \right\}.
\end{equation*}
Since $y_0 \in S_Y$ is a Daugavet point, we can find $v_0 \in S_2$ such that $\|y_0 + v_0\| \geq 2 - \eps$. Now, since $v_0 \in S_2$ and $u_0 \in S_1$, we have that
\begin{eqnarray*}
	G(u_0, v_0) &>& \sup_{v \in B_Y} G(u_0, v) - \frac{\alpha}{4} \\
	&>& \sup_{u \in B_X} G(u, y_1) - \frac{\alpha}{4} - \frac{\alpha}{4} \\
	&>& G(x_1, y_1) - \frac{\alpha}{2} \\
	&>& 1 - \alpha,	
\end{eqnarray*}
which shows that $u_0 \otimes v_0 \in S$. We will prove that $\|x_0 \otimes y_0 - u_0 \otimes v_0\| \geq 2(1 - \eps)^2$. Indeed, let $x^* \in S_{X^*}$ and $y^* \in S_{Y^*}$ be such that 
\begin{equation*}
x^*(x_0 - u_0) \geq 2 - \eps \ \ \ \mbox{and} \ \ \ y^*(y_0 + v_0) \geq 2 - \eps.
\end{equation*}
In particular, we have that $x^*(x_0), y^*(y_0)$, and $y^*(v_0)$ are all $\geq 1 - \eps$ and, on the other hand, that $x^*(u_0) \leq 1 - \eps$. Finally, define $G_0 \in \mathcal{B}(X \times Y)$ by $G_0(x, y) := x^*(x) y^*(y)$ for every $x \in X$ and $y \in Y$. Then, $\|G_0\| = 1$ and
\begin{eqnarray*}
	\|x_0 \otimes y_0 - u_0 \otimes v_0\| &\geq& G_0(x_0 \otimes y_0 - u_0 \otimes v_0) \\
	&=& x^*(x_0) y^*(y_0) - x^*(u_0) y^*(v_0) \\
	&\geq& (1 - \eps)^2 + (1 - \eps)^2 \\
	&=& 2 (1 - \eps)^2
\end{eqnarray*}
as we wanted. Therefore, $x_0 \otimes y_0 \in S_{X \pten Y}$ is a Daugavet point.
	
\noindent
(c). Let $x_0 \in S_X$ and $y_0 \in S_Y$ be such that $x_0 \otimes y_0 \in S_{X \pten Y}$ is a Daugavet point and let us assume that $y_0 \in S_Y$ is a denting point. To prove that $x_0$ is also a Daugavet point, take $S := S(B_X, x^*, \alpha)$ to be a slice of $B_X$ with $\|x^*\| = 1$ and $\alpha > 0$. Since $y_0$ is a denting point, we can take a slice $S_1 := S(B_Y, y^*, \beta)$ of $B_Y$ with $\|y^*\| = 1$, $\beta > 0$, and such that $y_0 \in S_1$ and $\diam(S_1) < \eps$. Let us define the bilinear form $G(x, y) := x^*(x) y^*(y)$ for every $x \in X$ and $y \in Y$. Now, consider the slice $S_2 := S(B_{X \pten Y}, G, \eta)$ of $B_{X \pten Y}$, where $\eta < \min\{\alpha, \beta\}$. Since $x_0 \otimes y_0 \in S_{X \pten Y}$ is a Daugavet point, by using Lemma \ref{lemmaslices}, we may find $u_0 \otimes v_0 \in S_2$ so that
\begin{equation} \label{ineq1} 
\|x_0 \otimes y_0 - u_0 \otimes v_0\| \geq 2 - \eps. 
\end{equation}
Since $\|y_0\| \leq 1$ and $\|v_0\| \leq 1$, (\ref{ineq1}) implies that 
\begin{equation*}
\|x_0 - u_0\| + \|y_0 - v_0\| \geq 2 - \eps. 
\end{equation*}
On the other hand, since $u_0 \otimes v_0 \in S_2$, we have that
\begin{equation*}
x^*(u_0) y^*(v_0) = G(u_0, v_0) > 1 - \eta = 1 - \min\{\alpha, \beta\}.
\end{equation*}
This implies, in particular, that $x^*(u_0) > 1 - \alpha$ and $y^*(v_0) > 1 - \beta$, that is, we have that $u_0 \in S$ and $v_0 \in S_1$. Since $\diam (S_1) < \eps$, we have that $\|y_0 - v_0\| < \eps$ and therefore
\begin{equation*}
2 - \eps \leq \|x_0 - u_0\| + \|y_0 - v_0\| < \|x_0 - u_0\| + \eps,
\end{equation*}
which implies that $\|x_0 - u_0\| > 2 - 2 \eps$. This shows that $x_0 \in S_X$ is a Daugavet point as desired.	
\end{proof}

\section{An example of a Banach space which satisfies the WODP} \label{sectionanexample}

Let us recall that a Banach space $X$ is said to be $L$-{\it embedded} if $X^{**} = X \oplus_1 Z$ for some subspace $Z$ of $X^{**}$. Examples of $L$-embedded spaces are $L_1(\mu)$-spaces, preduals of von Neumann algebras, preduals of real or complex JBW$^*$-triples or the disk algebra. Let us recall also that a Banach space $X$ is said to satisfy the {\it bounded approximation property} if there exists a positive constant $\lambda$ such that, for every compact subspace $K$ of $X$ and every $\eps > 0$, there exists $S \in \mathcal{F}(X)$ such that $\|S\| \leq \lambda$ and $\|x - S(x)\| \leq \eps$ for every $x \in K$. If this holds for $\lambda = 1$, then $X$ is said to have the {\it metric approximation property} (MAP, for short).

\begin{theorem} \label{WODP} Let $X$ be Banach space which is $L$-embedded and such that it satisfies both metric approximation and the Daugavet properties. Suppose further that $\dens(X) \leq \omega_1$. Then, $X$ has the WODP.
\end{theorem}

\begin{proof} Let $\eps > 0$ be given. Let us fix $x_1, \ldots, x_N \in S_X$, a slice $S := S(B_X, x^*, \eta)$ of $B_X$, and $x' \in B_X$. In order to prove that $X$ has the WODP, we will find a point $x_0 \in S$ and an operator $R \in \mathcal{L}(X)$ satisfying
	\begin{itemize}
		\setlength\itemsep{0.2cm}
		\item[(a)] $\|R\| \leq 1 + \eps$,
		\item[(b)] $\|R(x_i) - x_i\| < \eps$ for every $i = 1, \ldots, N$, and
		\item[(c)] $\|R(x_0) - x'\| < \eps$. 
	\end{itemize}
	Since $X$ is $L$-embedded, we may write $X^{**} = X \oplus_1 Z$ for some subspace $Z$ of $X$. Since $S(B_{X^{**}}, x^*, \eta)$ is a $w^*$-open subset of $B_{X^{**}}$, there exists $u \in S_{Z} \cap S(B_X, x^*, \eta)$, and henceforth
	\begin{equation*}
	\|x + u\| = 1 + \|x\|
	\end{equation*}
	holds for every $x \in X$ (the case $\dens(X)=\omega_0$ was proved in \cite[Theorem 3.3]{rueda} whereas the case $\dens(X)=\omega_1$ was proved in \cite[Theorem 4.1]{GR}).

	Using the hypothesis that $X$ has the MAP, we obtain $T \in \mathcal{F}(X)$ such that $\|T\| \leq 1$ and $\|T(x_i) - x_i\| < \eps$ for every $i = 1, \ldots, N$. Define $\widehat{T} \in \mathcal{L}( X^{**})$ by
	\begin{equation*}
	\widehat{T} (x + z) := T(x) + \varphi(z) x' \ \ \ (x \in X, z \in Z)
	\end{equation*}
	where $\varphi \in S_{X^{***}}$ is such that $\varphi(u) = 1$. Since $\|T\| \leq 1$ and $x' \in B_X$, we have that $\|\widehat{T}\| \leq 1$. Moreover, $\widehat{T}$ is a finite rank operator and $\widehat{T}(u) = x'$.

	We will apply twice \cite[Lemma 2.5]{OP} to find the desired operator $R$ which satisfies (a), (b), (c) above. Indeed, set $F := \Span \{ x_1, \ldots, x_N, u \}$, which is a finite-dimensional subspace of $X^{**}$. By \cite[Lemma 2.5]{OP}, there exists an operator $S \in \mathcal{F}(X^*)$ such that
	\begin{itemize}
	\setlength\itemsep{0.2cm} 
		\item[(i)] $| \|S\| - \| \widehat{T} \| | < \frac{\eps}{2}$,
		\item[(ii)] $\ran S^* = \ran \widehat{T}$, 
		\item[(iii)] $S^*(v) = \widehat{T}(v)$ for every $v \in F = \Span \{ x_1, \ldots, x_N, u \}$, and 
		\item[(iv)] $S^{**} x^{***} = (\widehat{T})^* x^{***}$ for every $x^{***} \in X^{***}$ for which $(\widehat{T})^*(x^{***}) \in X^*$.
	\end{itemize}
	Now, applying once again \cite[Lemma 2.5]{OP} for $S \in \mathcal{F}(X^*)$, we can find $R \in \mathcal{F}(X)$ such that the following further properties are satisfied:
	\begin{itemize}
		\setlength\itemsep{0.2cm}
		\item[(iv)] $| \|R\| - \|S\| | < \frac{\eps}{2}$,
		\item[(v)] $\ran R^* = \ran S$, and 
		\item[(vi)] $R^{**} x^{**} = S^* x^{**}$ for every $x^{**} \in X^{**}$ for which $S^* x^{**} \in X$. 
	\end{itemize}
	By (iii), we have in particular that $S^*(x_i) = \widehat{T}(x_i) = T(x_i) \in X$ for every $i = 1, \ldots, N$. Thus, by using now (vi), we get that $R(x_i) = R^{**}(x_i) = S^*(x_i)$ for every $i=1, \ldots, N$. It follows that $R(x_i) = T(x_i)$ for every $i = 1, \ldots, N$ and so (b) is proved. Now, by using (iv) and (i), we have that
	\begin{equation*}
	\|R\| < \|S\| + \frac{\eps}{2} < \|\widehat{T}\| + \eps \leq 1 + \eps
	\end{equation*}
	and this proves (a). It remains to prove (c), that is, there exists $x_0 \in S(B_X, x^*, \eta)$ such that $\|R(x_0) - x'\| < \eps$. Indeed, by using (iii) we have that $S^*(u) = \widehat{T}(u) = x' \in B_X$ and then, by (vi) again, we get that $R^{**}(u) = S^*(u) = x'$. Let $(u_{\alpha}) \subseteq B_X$ be such that $u_{\alpha} \stackrel{w^*}{\longrightarrow} u$. Passing to a subnet (if it is necessary), we may (and we do) assume that $u_{\alpha} \in S(B_X, x^*, \eta)$ for all $\alpha \in I$. On the other hand, since $R^{**}$ is $w^*$-$w^*$ continuous, then we have that
	\begin{equation*}
	R(u_{\alpha}) = R^{**}(u_{\alpha}) \stackrel{w^*}{\longrightarrow} R^{**}(u) = x'
	\end{equation*}
	Since $R(u_{\alpha}) \in X$ and $x' \in X$, we have that in fact $R(u_{\alpha}) \stackrel{w}{\longrightarrow} x'$. Therefore,
	\begin{equation*}
	x' \in \overline{ \{R(u_{\alpha}): \alpha \in I \} }^{w} \subseteq \overline{\co} \{ R (u_{\alpha}): \alpha \in I \} \subset \overline{R(\co \{u_{\alpha}: \alpha \in I\})}.
	\end{equation*}	
	This implies that there exists $x_0 \in \co \{u_{\alpha}: \alpha \in I\}$ such that $\|R(x_0) - x'\| < \eps$. Since $u_{\alpha} \in S(B_X, x^*, \eta)$ for every $\alpha \in I$, we conclude that $x_0 \in S(B_X, x^*, \eta)$ and this proves (c) and we are done.
\end{proof}



\begin{thebibliography}{99}
	
 \bibitem{AHLP} \textsc{T.A.~Abrahamsen, R.~Haller, V.~Lima, and K.~Pirk}, Delta- and Daugavet-points in Banach spaces, Proc. Edinburgh Math. Soc. {\bf 63} (2020), 475--496 
 
 \bibitem{ALMT} \textsc{T.A.~Abrahamsen, V.~Lima, A.~Martiny, and S.~Troyanski}, Daugavet- and $\Delta$-points in Banach spaces with unconditional bases, available on Arxiv, https://arxiv.org/abs/2007.04946
	
\bibitem{BLR} \textsc{J.~Becerra Guerrero, G.~L\'opez P\'erez, and A.~Rueda Zoca}, Octahedral norms and convex combination of slices in Banach spaces, J. Funct. Anal. {\bf 266} (2014), 2424--2436
 
	
	\bibitem {DF} \textsc{A.~Defant} and \textsc{K.~Floret}, \textit{Tensor Norms and Operator Ideals}, North-Holland, Mathematics Studies, Elsevier, 1993.
	
 
	
	\bibitem{FHHMPZ} \textsc{M.~Fabian}, \textsc{M.~Habala}, \textsc{P.~H\'ajek}, \textsc{V.~Montesinos Santaluc\'ia}, \textsc{J.~Pelant}, \textsc{V.~Zizler}, \textit{Functional Analysis and Infinite-Dimensional Spaces}, Springer, 2000.
	
	
\bibitem {GGMS} \textsc{N. Ghoussoub, G. Godefroy, B. Maurey, W. Schachermayer}, Some topological and geometrical
structures in Banach spaces, Memoirs Amer. Math. Soc.  {\bf
70} (378), 1987.	


	
\bibitem {HPV} 	\textsc{R.~Haller, K.~Pirk and T.~Veeorg}, Daugavet- and Delta-Points in Absolute Sums of Banach Spaces, J. Conv. Anal. \textbf{ 28} (2021), 41--54. 
	
	\bibitem{IK} \textsc{Y.~Ivakhno and V.~Kadets}, Unconditional sums of spaces with bad projections, Visn. Khark. Univ., Ser. Mat. Prykl. Mat. Mekh. {\bf 645} (54) (2004), 30--35
	
	\bibitem{JR} \textsc{M.~Jung and A.~Rueda Zoca}, Daugavet point and $\Delta$-points in Lipschitz-free spaces, available on Arxiv, https://arxiv.org/abs/2010.09357
	
		
	\bibitem{KSSW}	\textsc{V.~Kadets, R.V.~Shvidkoy, G.G.~Sirotkin, and D.~Werner}, Banach spaces with the Daugavet property, Trans. Am. Math. Soc. {\bf 352}, 2 (2000), 855--873.
	
	\bibitem{KKW} \textsc{V.~Kadets, N.~Kalton, and D.~Werner}, Remarks on rich subspaces of Banach spaces, Studia Math. {\bf 159} (2003), 195--206
	
	\bibitem{LLR} \textsc{J.~Langemets, V.~Lima, and A.~Rueda Zoca}, Octahedral norms in tensor products of Banach spaces, Q. J. Math. 68, {\bf 4} (2017), 1247--1260.
	
	\bibitem{LP} \textsc{J.~Langemets and K.~Pirk}, Stability of diametral diameter two properties, available on Arxiv, https://arxiv.org/abs/2012.09492
	
	\bibitem{LR} \textsc{J.~Langemets and A.~Rueda Zoca}, Octahedral norms in duals and biduals of Lipschitz-free spaces, J. Funct. Anal. {\bf 279}, 3 (2020), 108557
	
	\bibitem{L} \textsc{J.~Lindenstrauss}, \textit{Extension of compact operators}, Mem. Amer. Math. Soc. \textbf{48} (1964)
 
 
	\bibitem{GR} \textsc{G.~L\'opez-P\'erez and A.~Rueda Zoca}, $L$-orthogonality, octahedrality and Daugavet property in Banach spaces, available on Arxiv, https://arxiv.org/abs/1912.09039 
 
	
 \bibitem{MR} \textsc{M.~Mart\'in and A.~Rueda Zoca}, Daugavet property in projective symmetric tensor products of Banac h spaces, preprint, available on Arxiv, https://arxiv.org/abs/2010.15936

	 \bibitem{OP} \textsc{E.~Oja and M. Põldvere}, Principle of Local Reflexivity Revisited, Proc. Amer. Math. Soc., {\bf135}, no. 4, (2007), 1081--1088.
	 
\bibitem {rueda} A. Rueda Zoca, \textit{Daugavet property and separability in Banach spaces}, Banach J. Math. Anal. \textbf{68}, 1 (2018), 68--84.	 
	 
 \bibitem {RTV} \textsc{A.~Rueda Zoca, P.~Tradacete, and I.~Villanueva}, Daugavet property in tensor product spaces, to appear in J. Inst. Math. Jussieu. 
	
	\bibitem {rya} \textsc{R.~A.~Ryan}, \emph{Introduction to tensor products of Banach spaces}, Springer Monographs in Mathematics, Springer-Verlag, London, 2002.
	
	\bibitem{S} \textsc{R.V.~Shvydkoy}, Geometric aspects of the Daugavet property, J. Funct. Anal., {\bf 176}, 2 (2000), 198--212
	

 \bibitem{W} \textsc{D.~Werner}, Recent progress on the Daugavet property, Irish Math. Soc. Bull. {\bf 46} (2001), 77-97
 		
\end{thebibliography}
\end{document}